\documentclass{amsart}
\usepackage{latexsym, amssymb, amsmath}

\newtheorem{conj}{Conjecture}
\newtheorem{thm}{Theorem}[section]
\newtheorem{lem}[thm]{Lemma}
\newtheorem{cor}[thm]{Corollary}
\newtheorem{prop}[thm]{Proposition}
\newtheorem{rem}[thm]{Remark}
\theoremstyle{definition}

\newcommand{\onabla}{\overline\nabla}

\newcommand{\p}{\phi}

\newcommand{\met}{\langle \cdot , \cdot \rangle}

\title{Biharmonic maps from a complete Riemannian manifold into a non-positively curved manifold}

\author{Shun Maeta}

\thanks{}

\address{\footnotesize{Faculty of Tourism and Business Management Shumei University, Chiba 276-0004, Japan.}
 }
\footnotesize{
\email{shun.maeta@gmail.com~{\it or}~maeta@mailg.shumei-u.ac.jp}
}

\begin{document} 
\maketitle 
\markboth{Biharmonic maps into a non-positively curved manifold} 
{Shun Maeta} 

\begin{abstract} 
We consider biharmonic maps $\phi:(M,g)\rightarrow (N,h)$ from a complete Riemannian manifold into a Riemannian manifold with non-positive sectional curvature.
Assume that $\alpha$ satisfies $1<\alpha<\infty$.
If for such an $\alpha$,
$\int_M|\tau(\phi)|^{\alpha}dv_g<\infty$
and 
$\int_M|d\phi|^2dv_g<\infty,$
 where $\tau(\phi)$ is the tension field of $\phi$,
then we show that $\phi$ is harmonic.
For a biharmonic submanifold, we obtain that the above assumption $\int_M|d\phi|^2dv_g<\infty$ is not necessary.
These results give affirmative partial answers to the global version of generalized Chen's conjecture.
\end{abstract}

\qquad\\


\section{Introduction}\label{intro} 
The theory of harmonic maps has been applied into various fields in differential geometry.
 Harmonic maps between two Riemannian manifolds are
 critical points of the {\em energy} functional 
 $E(\p)=\frac{1}{2}\int_M|d\p|^2dv_g,$
  for smooth maps $\p:(M^m,g)\rightarrow (N^n,h)$ from an $m$-dimensional Riemannian manifold into an $n$-dimensional Riemannian manifold, 
 where $dv_g$ denotes the volume element of $g.$
The Euler-Lagrange equation of $E$ is $\tau(\p)={\rm Trace}\nabla d\p=0,$
where $\tau(\p)$ is called the {\em tension field} of $\p.$ 
A map $\p:(M,g)\rightarrow (N,h)$ is called a {\em harmonic map} if $\tau(\p)=0.$

In 1983, J. Eells and L. Lemaire \cite{jell1} proposed the problem to consider biharmonic maps which are critical points of the {\em bi-energy} functional 
$E_2 (\phi )=\frac{1}{2}\int_M |\tau (\phi)| ^2 dv_g,$
 on the space of smooth maps between two Riemannian manifolds.
Biharmonic maps are, by definition, a generalization of harmonic maps.
In 1986, G. Y. Jiang \cite{jg1} derived the first and the second variational formulas of the bi-energy and studied biharmonic maps. The Euler-Lagrange equation of $E_2$ is 
$$\tau_2(\phi)=-\Delta^{\phi} \tau (\phi ) -\sum^m_{i=1} R^N (\tau (\phi )  , d\phi (e_i))d\phi (e_i)=0,$$
  where $\Delta^{\phi}:=\displaystyle\sum^m_{i=1}\left(\onabla_{e_i}\onabla_{e_i}-\onabla_{\nabla_{e_i}e_i}\right)$, and $\onabla$ is the induced connection on $\p^{-1}TN$.
 A map $\p:(M,g)\rightarrow(N,h)$ is a {\em biharmonic map} if $\tau_2(\p)=0$.  

One of the most interesting problem in the biharmonic theory is Chen's conjecture. 
 In 1988, B. Y. Chen raised the following problem:
\vspace{5pt}

\begin{conj}
[\cite{byc1}]\label{Chen}
Any biharmonic submanifold in $\mathbb{E}^n$ is minimal. 
\end{conj}
\vspace{5pt}
Here,
if $\p:(M,g)\rightarrow (N,h)$ is a biharmonic isometric immersion, then $M$ is called a {\em biharmonic submanifold} in $N$.

There are many affirmative partial answers to Chen's conjecture.
Chen's conjecture is solved completely if $M$ is one of the following: \\
 (a) a curve (cf. \cite{Dimi}),\\
 (b) a surface in $\mathbb{E}^3$ (cf. \cite{byc1}), \\
 (c) a hypersurface in $\mathbb{E}^4$ (cf. \cite{Leuven},\ \cite{Hasanis-Vlachos}). \\
 However, we cannot apply the methods in \cite{byc1}, \cite{Leuven} and \cite{Hasanis-Vlachos} to a higher dimensional manifold.
 In this paper, we try to think about Chen's conjecture from a different point of view.
 Here we notice that since there is no assumption of completeness for submanifolds in Chen's conjecture, in a sense it is a problem in local differential geometry.  
 With these understandings, we reformulate Chen's conjecture into a problem in global differential geometry (cf. \cite{kasm1},\ \cite{N-U-1}):
 
 \vspace{5pt}

\begin{conj}
\label{Akutagawa Maeta}
 Any complete biharmonic submanifold in $\mathbb{E}^n$ is minimal.
 \end{conj}

\vspace{5pt}

 On the other hand, Chen's conjecture was generalized as follows (cf. \cite{rcsmpp1}):
 "Any biharmonic submanifold in a Riemannian manifold with non-positive sectional curvature is  minimal."
 There are also many affirmative partial answers to this conjecture. \\
 (a) Any biharmonic submanifold in $\mathbb{H}^3(-1)$ is minimal (cf. \cite{rcsmco2}). \\
 (b) Any biharmonic hypersurface in $\mathbb{H}^4(-1)$ is minimal (cf. \cite{absmco2}).\\
 (c) Any compact biharmonic submanifold in a Riemannian manifold with non-positive sectional curvature is minimal (cf. \cite{jg1}).\\
 However, Y.-L. Ou and L. Tang gave a counterexample of this conjecture (cf. \cite{ylolt1}).
 Note that there are non-minimal, biharmonic submanifolds in a sphere (cf. \cite{jg1}).
 With these understandings, it is natural to consider the following problem.

 \vspace{5pt}

\begin{conj}\label{N-U}
 Any complete biharmonic submanifold in a Riemannian manifold with non-positive sectional curvature is minimal. 
\end{conj}

\vspace{5pt}

  K. Akutagawa and the author gave an affirmative partial answer to
 Conjecture~$\ref{Chen}$ (Conjecture~$\ref{Akutagawa Maeta}$ particularly) as follows (cf.~\cite{kasm1},~\cite{sm7}):
 
\vspace{5pt}

\begin{thm}
[\cite{kasm1}]
Any biharmonic proper submanifold in $\mathbb{E}^n$ is minimal.
\end{thm}

\vspace{5pt}

Here, an immersed submanifold $M$ in a Riemannian manifold $N$ is said to be {\em proper} if the immersion is a proper map.
  Note that the properness of the immersion implies the completeness of $(M, g)$.  
The author also gave an affirmative partial answer to Conjecture $\ref{N-U}$ (cf. \cite{sm9}).
N. Nakauchi and H. Urakawa also gave affirmative partial answers to Conjecture $\ref{N-U}$ (cf. \cite{N-U-1}, \cite{N-U-2}).

For biharmonic maps, N. Nakauchi, H. Urakawa and S. Gudmundsson showed the following important result (cf. \cite{nnhusg1}).

\vspace{5pt}

\begin{thm}[\cite{nnhusg1}]\label{Th of N-U-G}
Biharmonic maps from a complete Riemannian manifold into a non-positively curved manifold with finite bi-energy and energy is harmonic.
\end{thm}

\vspace{5pt}

One of the main result of this paper is the following generalization of Theorem $\ref{Th of N-U-G}$ as follows (cf. Theorem $\ref{Th map}$).

\vspace{5pt}

\begin{thm}[cf. Theorem $\ref{Th map}$]
Let $\p:(M,g)\rightarrow (N,h)$ be a biharmonic map from  a complete Riemannian manifold $(M,g)$ into a Riemannian manifold $(N,h)$ with non-positive sectional curvature
and let $\alpha$ be a real constant satisfying $1<\alpha<\infty$.

$(i)$ If
$$\int_M|\tau(\p)|^{\alpha}dv_g<\infty,$$
and the energy is finite, that is, 
$$\int_M|d\p|^2dv_g<\infty,$$
then $\phi$ is harmonic.

$(ii)$ If ${\rm Vol}(M,g)=\infty$ and
$$\int_M|\tau(\p)|^{\alpha}dv_g<\infty,$$
then $\phi$ is harmonic.
\end{thm}

\vspace{5pt}

If a biharmonic map $\p:(M,g)\rightarrow (N,h)$ is an isometric immersion, that is, $M$ is a biharmonic submanifold in $N$, then we have $\tau(\p)=m{\bf H}$, where ${\bf H}$ is the mean curvature vector field of $M.$
 For biharmonic submanifolds, we obtain that the assumption $\int_M|d\p|^2dv_g<\infty$
 in the above theorem is not necessary. 
 \vspace{5pt}

\begin{thm}[cf. Theorem $\ref{Th sub mfd}$]
Let $\p:(M,g)\rightarrow (N,h)$ be a biharmonic isometric immersion from  a complete Riemannian manifold $(M,g)$ into a Riemannian manifold $(N,h)$ with non-positive sectional curvature 
and let $\alpha$ be a real constant satisfying $1<\alpha<\infty$.

If 
$$\int_M|{\bf H}|^{\alpha}dv_g<\infty,$$
then $\phi$ is harmonic.
\end{thm}

\begin{rem}
In the above theorem, in case of codimension one, we only need the weaker assumption, non-positivity of the Ricci curvature of $(N,h)$.
\end{rem}

\vspace{5pt}

These results give affirmative partial answers to global version of generalized Chen's conjecture (Conjecture $\ref{N-U}$).

The remaining sections are organized as follows. 
Section~$\ref{Pre}$ contains some necessary definitions and preliminary geometric results.
 In section~$\ref{map}$, we consider a biharmonic map and isometric immersion from a complete Riemannian manifold into a non-positively curved manifold.
 In section~$\ref{sub}$, we apply the result of Section~$\ref{map}$ to a biharmonic submersion.

\qquad\\

\section{Preliminaries}\label{Pre} 

In this section, we shall give the definitions of harmonic maps and biharmonic maps.
We also recall biharmonic submanifolds.

Let $(M,g)$ be an $m$-dimensional Riemannian manifold and $(N,h)$, an $n$-dimensional Riemannian manifold, respectively.
We denote by $\nabla$ and $\nabla^N$, the Levi-Civita connections on $(M,g)$ and $(N,h)$, respectively and by $\onabla$ the induced connection on $\p^{-1}TN$.

\vspace{10pt}

Let us recall the definition of a harmonic map $\p:(M,g)\rightarrow (N,h)$.
For a smooth map $\phi:(M,g)\rightarrow (N,h)$, the {\em energy} of $\phi$ is defined by
$$E(\phi) =\frac{1}{2}\int_M|d\phi|^2 dv_g.$$
The Euler-Lagrange equation of $E$ is 
$$\tau(\p)=\displaystyle \sum^m_{i=1}\{\onabla_{e_i}d\p(e_i)-d\p(\nabla_{e_i}e_i)\}=0,$$
where $\tau(\p)$ is called the {\em tension field} of $\p$ and $\{e_i\}_{i=1}^m$ is an orthonormal frame field on $M$.
 A map $\p:(M,g)\rightarrow (N,h)$ is called a {\em harmonic map} if $\tau(\p)=0$. 

\vspace{10pt}

In 1983, J. Eells and L. Lemaire \cite{jell1} proposed the problem to consider biharmonic maps which are critical points of the bi-energy functional on the space of smooth maps between two Riemannian manifolds.
In 1986, G. Y. Jiang \cite{jg1} derived the first and the second variational formulas of bi-energy and studied biharmonic maps.
For a smooth map $\phi:(M,g)\rightarrow (N,h)$, the  {\em bi-energy} of $\phi$ is defined by
$$E_2 (\phi )=\frac{1}{2}\int_M |\tau (\phi)| ^2 dv_g.$$
The Euler-Lagrange equation of $E_2$ is 
\begin{equation}\label{NSbi}
\tau_2(\phi)=-\Delta^{\phi} \tau (\phi ) -\sum^m_{i=1} R^N (\tau (\phi )  , d\phi (e_i))d\phi (e_i)=0,
\end{equation}
where $\tau_2(\p)$ is called the {\em bi-tension field} of $\p$ and $R^N$ is the Riemannian curvature tensor of $(N,h)$ given by $R^N(X,Y)Z=\nabla^N_X\nabla^N_YZ-\nabla^N_Y\nabla^N_XZ-\nabla^N_{[X,Y]}Z$ for $X,\ Y,\ Z\in \frak{X}(N)$.
 A map $\p:(M,g)\rightarrow (N,h)$ is called a {\em biharmonic map} if $\tau_2(\p)=0$. 

\vspace{10pt}

We also recall biharmonic submanifolds.

Let $\p:(M^m,g)\rightarrow (N^n,h=\met)$ be an isometric immersion from an $m$-dimensional Riemannian manifold into an $n$-dimensional Riemannian manifold.
In this case, we identify $d\p(X)$ with $X\in \frak{X}(M)$ for each $x\in M.$
We also denote by $\met$ the induced metric $\p^{-1}h$.
The Gauss formula is given by
\begin{equation}
\nabla^N_XY=\nabla _XY+B(X,Y),\ \ \ \ X,Y\in \frak{X}(M),
\end{equation}
where $B$ is the second fundamental form of $M$ in $N$.
The Weingarten formula is given by
\begin{equation}\label{2.Wformula}
\nabla^N_X \xi =-A_{\xi}X+\nabla^{\perp}_X{\xi},\ \ \ X\in \frak{X}(M),\  \xi \in \frak{X}(M)^{\perp},  
\end{equation}
where $A_{\xi}$ is the shape operator for a unit normal vector field $\xi$ on $M,$ and $\nabla^{\perp}$ denotes the normal connection on the normal bundle of $M$ in $N$.
It is well known that $B$ and $A$ are related by
\begin{equation}\label{2.BA rel}
\langle B(X,Y), \xi \rangle=\langle A_{\xi}X,Y \rangle.
\end{equation}

For any $x \in M$, let $\{e_1, \cdots, e_m, e_{m+1}, \cdots, e_n\}$ be an orthonormal basis of $N$ at $x$ such that $\{e_1, \cdots, e_m\}$ is an orthonormal basis of $T_xM$. 
Then, $B$ is decomposed as 
$$ 
B(X, Y) = \sum_{\alpha=m+1}^n B_{\alpha}(X, Y)e_{\alpha},~~{\rm at}~x. 
$$ 
The mean curvature vector field ${\bf H}$ of $M$ at $x$ is also given by 
$$ 
{\bf H}(x) = \frac{1}{m} \sum_{i = 1}^m B(e_i, e_i) =\sum_{\alpha=m+1}^n H_{\alpha}(x)e_{\alpha},\qquad 
H_{\alpha}(x) := \frac{1}{m} \sum_{i = 1}^m B_{\alpha}(e_i, e_i).  
$$ 

If an isometric immersion $\p:(M,g)\rightarrow (N,h)$ is biharmonic, then $M$ is called a {\em biharmonic submanifold} in $N$.
 In this case, we remark that the tension field $\tau(\p)$ of $\p$ is written as $\tau(\p)=m{\bf H}$, where ${\bf H}$ is the mean curvature vector field of $M$.
The necessary and sufficient condition for $M$ in $N$ to be biharmonic is the following: 
\begin{equation}\label{NS bih sub}
\Delta^{\p}{\bf H}+\sum_{i=1}^mR^N({\bf H},d\p(e_i))d\p(e_i)=0.
\end{equation}

\qquad\\


\section{Biharmonic maps into non-positively curved manifolds}\label{map} 
In this section, we shall show the following theorem.

\vspace{5pt}

\begin{thm}
\label{Th map}
Let $\p:(M,g)\rightarrow (N,h)$ be a biharmonic map from  a complete Riemannian manifold $(M,g)$ into a Riemannian manifold $(N,h)$ with non-positive sectional curvature
and let $\alpha$ be a real constant satisfying $1<\alpha<\infty$.

$(i)$ If
$$\int_M|\tau(\p)|^{\alpha}dv_g<\infty,$$
and the energy is finite, that is, 
$$\int_M|d\p|^2dv_g<\infty,$$
then $\phi$ is harmonic.

$(ii)$ If ${\rm Vol}(M,g)=\infty$ and
$$\int_M|\tau(\p)|^{\alpha}dv_g<\infty,$$
then $\phi$ is harmonic.

\end{thm}

\vspace{5pt}

 Before proving Theorem $\ref{Th map}$, we shall show  the following lemma.

\begin{lem}\label{key lem 1}
Let $\p:(M,g)\rightarrow (N,h)$ be a biharmonic map from  a complete Riemannian manifold $(M,g)$ into a Riemannian manifold $(N,h)$ with non-positive sectional curvature.

Assume that $\alpha$ satisfies $1<\alpha<\infty$.
If for such an $\alpha$,
$$\int_M|\tau(\p)|^{\alpha}dv_g<\infty,$$
then $\onabla_X\tau(\p)=0$ for any vector field $X$ on $M$. In particular, $|\tau(\p)|$ is constant.
\end{lem}

\begin{proof}
For a fixed point $x_0\in M$, and for every $0<r<\infty,$
 we first take a cut off function $\lambda$ on $M$ satisfying that 
 \begin{equation}
\left\{
 \begin{aligned}
&0\leq\lambda(x)\leq1\ \ \ (x\in M),\\
&\lambda(x)=1\ \ \ \ \ \ \ \ \ (x\in B_r(x_0)),\\
&\lambda(x)=0\ \ \ \ \ \ \ \ \ (x\not\in B_{2r}(x_0)),\\
&|\nabla\lambda|\leq\frac{C}{r}\ \ \ \ \ \ \ (x\in M),\ \ \ \text{for some constant $C$ independent of $r$},
\end{aligned} 
\right.
\end{equation}
where $B_r(x_0)$ and  $B_{2r}(x_0)$ are the balls centered at a fixed point $x_0\in M$ with radius $r$ and $2r$ respectively (cf. \cite{ka1}).
From $(\ref{NSbi})$, we have
\begin{equation}\label{0}
\begin{aligned}
&\int_M\langle -\Delta^{\phi}\tau(\p),\lambda^2|\tau(\p)|^{\alpha-2}\tau(\p)\rangle dv_g\\
&=\int_M\lambda^2|\tau(\p)|^{\alpha-2}\sum^m_{i=1}\langle R^N(\tau(\p),d\p(e_i))d\p(e_i),\tau(\p)\rangle dv_g\leq0,
\end{aligned}
\end{equation}
where the inequality follows from the sectional curvature of $(N,h)$ is non-positive.
By $(\ref{0})$, we have
\begin{equation}\label{branch}
\begin{aligned}
0\geq&\int_M\langle -\Delta^{\phi}\tau(\p),\lambda^2|\tau(\p)|^{\alpha-2}\tau(\p)\rangle dv_g\\
=&\int_M\langle \onabla\tau(\p),\onabla(\lambda^2|\tau(\p)|^{\alpha-2}\tau(\p))\rangle dv_g\\
=&\int_M\sum_{i=1}^m\langle \onabla_{e_i}\tau(\p),(e_i\lambda^2)|\tau(\p)|^{\alpha-2}\tau(\p)
+\lambda^2e_i\{(|\tau(\p)|^2)^{\frac{\alpha-2}{2}}\}\tau(\p)\\
&\hspace{180pt}+\lambda^2|\tau(\p)|^{\alpha-2}\onabla_{e_i}\tau(\p)\rangle dv_g\\
=&\int_M\sum_{i=1}^m\langle \onabla_{e_i}\tau(\p), 2\lambda(e_i\lambda)|\tau(\p)|^{\alpha-2}\tau(\p)\rangle dv_g\\
&+\int_M\sum_{i=1}^m\langle \onabla_{e_i}\tau(\p),\lambda^2(\alpha-2)|\tau(\p)|^{\alpha-4}\langle\onabla_{e_i}\tau(\p),\tau(\p)\rangle \tau(\p)\rangle dv_g\\
&+\int_M\sum_{i=1}^m\langle \onabla_{e_i}\tau(\p), \lambda^2 |\tau(\p)|^{\alpha-2} \onabla_{e_i}\tau(\p)\rangle dv_g.
\end{aligned}
\end{equation}
(i) The case of $\alpha<2$.
We have
\begin{equation}\label{ast}
\begin{aligned}
&\int_M\sum_{i=1}^m \lambda^2|\tau(\p)|^{\alpha-2}|\onabla_{e_i}\tau(\p)|^2 dv_g\\
&\leq-2\int_M\sum_{i=1}^m\langle \onabla_{e_i}\tau(\p), \lambda (e_i\lambda) |\tau(\p)|^{\alpha-2}\tau(\p)\rangle dv_g\\
&\ \ \ -(\alpha-2)\int_M\sum_{i=1}^m\langle \onabla_{e_i}\tau(\p),\lambda^2|\tau(\p)|^{\alpha-4}\langle \onabla _{e_i}\tau(\p),\tau(\p) \rangle \tau(\p) \rangle dv_g.
\end{aligned}
\end{equation}

We shall consider the first term of the right hand side of $(\ref{ast})$.

\begin{equation}\label{1}
\begin{aligned}
-2&\int_M\sum_{i=1}^m\langle \onabla_{e_i}\tau(\p),\lambda (e_i \lambda)|\tau(\p)|^{\alpha-2}\tau(\p)\rangle dv_g\\
=&-2\int_M\sum_{i=1}^m\langle \lambda |\tau(\p)|^{\frac{\alpha}{2}-1} \onabla_{e_i}\tau(\p),
(e_i \lambda ) |\tau(\p)|^{\frac{\alpha}{2}-1}\tau(\p)\rangle dv_g\\
\leq&~\varepsilon \int_M \sum_{i=1}^m \lambda^2 |\tau(\p)|^{\alpha-2} |\onabla_{e_i}\tau(\p)|^2dv_g\\
&+\frac{1}{\varepsilon}\int_M|\nabla\lambda|^2|\tau(\p)|^{\alpha}dv_g,
\end{aligned}
\end{equation} 
where the inequality of $(\ref{1})$ follows from the following inequality
\begin{equation}\label{Young's ineq.}
\pm2 \langle V, W\rangle \leq \varepsilon |V|^2+\frac{1}{\varepsilon}|W|^2,\ \ \ \ \ \text{for all positive}\  \varepsilon >0,
\end{equation}
because of the inequality $0\leq |\sqrt{\varepsilon}V\pm \frac{1}{\sqrt{\varepsilon}}W|^2$.
The inequality $(\ref{Young's ineq.})$ is called {\em Young's inequality}.

 We shall consider the second term of the right hand side of $(\ref{ast})$.

\begin{equation}\label{2}
\begin{aligned}
-(\alpha-2)&\int_M\sum_{i=1}^m\langle \onabla_{e_i}\tau(\p),\lambda^2|\tau(\p)|^{\alpha-4}\langle \onabla_{e_i}\tau(\p),\tau(\p)\rangle \tau(\p)\rangle dv_g\\
\leq&~(2-\alpha)\int_M\sum_{i=1}^m|\onabla_{e_i}\tau(\p)|^2\lambda^2|\tau(\p)|^{\alpha-2}dv_g,
\end{aligned}
\end{equation}
where the inequality of $(\ref{2})$ follows from Cauchy-Schwartz inequality.

Substituting $(\ref{1})$ and $(\ref{2})$ into $(\ref{ast})$, we have

\begin{equation}
\begin{aligned}
&\int_M\sum_{i=1}^m\lambda^2|\tau(\p)|^{\alpha-2}|\onabla_{e_i}\tau(\p)|^2dv_g\\
&\leq~\varepsilon \int_M\sum_{i=1}^m \lambda^2|\tau(\p)|^{\alpha-2} |\onabla_{e_i}\tau(\p)|^2dv_g\\
&\ \ \ \ \ \ \ \ +\frac{1}{\varepsilon}\int_M|\nabla \lambda|^2|\tau(\p)|^{\alpha}dv_g\\
&\ \ \ \ \ \ \ \ +(2-\alpha) \int_M\sum_{i=1}^m|\onabla_{e_i}\tau(\p)|^2\lambda^2|\tau(\p)|^{\alpha-2}dv_g.
\end{aligned}
\end{equation}
Thus we have

\begin{equation}\label{ast2}
\begin{aligned}
(-1-\varepsilon+\alpha)&\int_M\sum_{i=1}^m\lambda^2|\tau(\p)|^{\alpha-2}|\onabla_{e_i}\tau(\p)|^2dv_g\\
&\leq~\frac{1}{\varepsilon}\int_M|\nabla\lambda|^2|\tau(\p)|^{\alpha}dv_g\\
&\leq~\frac{1}{\varepsilon}\frac{C^2}{r^2}\int_M|\tau(\p)|^{\alpha}dv_g.
\end{aligned}
\end{equation}
Since $(M,g)$ is complete, we tend $r$ to infinity.
 By the assumption $\int_M|\tau(\p)|^{\alpha}dv_g<\infty$,
  the right hand side of $(\ref{ast2})$ goes to zero and the left hand side of $(\ref{ast2})$ goes to 
  $$(-1-\varepsilon+\alpha)\int_M\sum_{i=1}^m|\tau(\p)|^{\alpha-2}|\onabla_{e_i}\tau(\p)|^2dv_g,$$ since $\lambda=1$ on $B_r(x_0)$.
  Since we can take that $\varepsilon$ is sufficiently small, by the assumption  $1<\alpha$, we have 
  $$\int_M\sum_{i=1}^m|\tau(\p)|^{\alpha-2}|\onabla_{e_i}\tau(\p)|^2dv_g=0.$$
   From this, we obtain for any vector field $X$ on $M$,
  \begin{equation}\label{eq. for constant}
  \onabla_X\tau(\p)=0.
  \end{equation}
 By $(\ref{eq. for constant})$,
  $$X|\tau(\p)|^2=2\langle \onabla_X\tau(\p),\tau(\p)\rangle=0.$$
  Therefore we obtain $|\tau(\p)|$ is constant. 
  
  (ii) The case $\alpha \geq2$.
  From $(\ref{branch})$, we have
 \begin{equation}\label{i-1}
 \begin{aligned}
 0\geq 
 &\int_M\sum_{i=1}^m\langle \onabla_{e_i} \tau(\p), 2\lambda(e_i\lambda)| \tau(\p)|^{\alpha-2} \tau(\p)\rangle dv_g\\
&+\int_M\sum_{i=1}^m\langle \onabla_{e_i} \tau(\p), \lambda^2 | \tau(\p)|^{\alpha-2} \onabla_{e_i} \tau(\p)\rangle dv_g.
 \end{aligned}
 \end{equation}
By using Young' inequality, that is, 
$$\pm2\langle V,W \rangle \leq \varepsilon |V|^2+\frac{1}{\varepsilon}|W|^2,$$
for all positive $\varepsilon$, we have
\begin{equation*} 
\begin{aligned}
 &-2\int_M\sum_{i=1}^m\langle \onabla_{e_i} \tau(\p), \lambda(e_i\lambda)| \tau(\p)|^{\alpha-2} \tau(\p)\rangle dv_g\\
=&-2\int_M\sum_{i=1}^m\langle (e_i\lambda) | \tau(\p)|^{\frac{\alpha}{2}-1} \tau(\p),
\lambda | \tau(\p)|^{\frac{\alpha}{2}-1} \onabla_{e_i} \tau(\p)\rangle dv_g\\
\leq&~2\int_M|\nabla\lambda|^2| \tau(\p)|^{\alpha}dv_g\\
&+\frac{1}{2}\int_M\lambda^2|\tau(\p)|^{\alpha-2}|\onabla_{e_i}\tau(\p)|^2dv_g.
\end{aligned}
\end{equation*}
Substituting this into $(\ref{i-1})$, we have
\begin{equation}\label{ast20}
\begin{aligned}
\int_M\lambda^2|\tau(\p)|^{\alpha-2}|\onabla_{e_i}\tau(\p)|^2dv_g
\leq&~4\int_M|\nabla\lambda|^2 | \tau(\p)|^{\alpha}dv_g\\
\leq&\int_M \frac{4C^2}{r^2}| \tau(\p)|^{\alpha}dv_g.
\end{aligned}
\end{equation}
 By the assumption $\int_M| \tau(\p)|^{\alpha}dv_g<\infty$,
  the right hand side of $(\ref{ast20})$ goes to zero and the left hand side of $(\ref{ast20})$ goes to 
  $$\int_M|\tau(\p)|^{\alpha-2}|\onabla_{e_i}\tau(\p)|^2dv_g,$$
   since $\lambda=1$ on $B_r(x_0)$.
 Thus, we have
  $$\int_M|\tau(\p)|^{\alpha-2}|\onabla_{e_i}\tau(\p)|^2dv_g=0.$$
   From this, we obtain for any vector field $X$ on $M$,
  \begin{equation}\label{eq. for constant0}
  \onabla_X \tau(\p)=0.
  \end{equation}
 By $(\ref{eq. for constant0})$,
  $$X| \tau(\p)|^2=2\langle \onabla_X \tau(\p), \tau(\p)\rangle=0.$$
  Therefore we obtain $| \tau(\p)|$ is constant. 

\end{proof}

To prove Theorem $\ref{Th map}$, we recall Gaffney's theorem (cf. \cite{mpg1}).

\vspace{5pt}

\begin{thm}
[\cite{mpg1}]\label{Gaffney}
Let $(M,g)$ be a complete Riemannian manifold. 
If a $C^1$ 1-form $\omega$ satisfies that $\int_M|\omega|dv_g<\infty$ and $\int_M(\delta \omega)dv_g<\infty$, or equivalently, a $C^1$ vector field $X$ defined by $\omega (Y)=\langle X,Y\rangle $,\ \ $(\forall ~Y\in \frak{X}(M))$ satisfies that $\int_M|X|dv_g<\infty$ and $\int_M{\rm div}(X) dv_g<\infty$, then
$$\int_M(\delta \omega)dv_g=\int_M{\rm div}(X)dv_g=0.$$
\end{thm}

\vspace{5pt}

By using Lemma $\ref{key lem 1}$ and Theorem $\ref{Gaffney}$, we shall show Theorem $\ref{Th map}$.

\begin{proof}[Proof of Theorem $\ref{Th map}$]
By Lemma ${\ref{key lem 1}}$, we have $\onabla _{X}\tau(\p)=0$ for any vector field $X$ on $M,$ and $|\tau(\p)|$ is constant.

We shall show the case $(ii)$. If ${\rm Vol}(M,g)=\infty$ and $|\tau(\p)|\not=0$, then
$$\int_M|\tau(\p)|^{\alpha}dv_g=|\tau(\p)|^{\alpha}{\rm Vol}(M,g)=\infty,$$
 which yields the contradiction.
 
We shall show the case $(i)$.
 Define a 1-form $\omega$ on $M$ by
$$\omega(X):=|\tau(\p)|^{\frac{\alpha}{2}-1}\langle d\p(X),\tau(\p)\rangle,\ \ \ \ \ (X\in \frak{X}(M)).$$
By the assumption $\int_M|d\p|^2dv_g<\infty$ and $\int_M|\tau(\p)|^{\alpha}dv_g<\infty$,
 we have

 \begin{equation}\label{Assumption of Gaffney 1}
 \begin{aligned}
 \int_M|\omega|dv_g
 =&\int_M\left(\sum_{i=1}^m|\omega(e_i)|^2\right)^{\frac{1}{2}}dv_g\\
 \leq&~\int_M|\tau(\p)|^{\frac{\alpha}{2}}|d\p|dv_g\\
 \leq&\left(\int_M|d\p|^2dv_g\right)^{\frac{1}{2}}\left(\int_M|\tau(\p)|^{\alpha}dv_g\right)^{\frac{1}{2}}<\infty.
 \end{aligned}
 \end{equation}
 We consider $-\delta \omega=\displaystyle\sum_{i=1}^m(\nabla_{e_i}\omega)(e_i)$.

 \begin{equation}
 \begin{aligned}
 -\delta \omega
 =&\sum_{i=1}^m\nabla_{e_i}(\omega(e_i))-\omega(\nabla_{e_i}e_i)\\
 =&\sum_{i=1}^m\Big\{\nabla_{e_i}\Big(|\tau(\p)|^{\frac{\alpha}{2}-1}\langle d\p(e_i),\tau(\p)\rangle\Big)\\
&\hspace{30pt}-|\tau(\p)|^{\frac{\alpha}{2}-1}\langle d\p(\nabla_{e_i}e_i),\tau(\p)\rangle\Big\}\\
=&\sum_{i=1}^m\Big\{ |\tau(\p)|^{\frac{\alpha}{2}-1}\langle \onabla_{e_i}d\p(e_i), \tau(\p)\rangle\\
&\hspace{30pt}-|\tau(\p)|^{\frac{\alpha}{2}-1}\langle d\p(\nabla_{e_i}e_i), \tau(\p)\rangle\Big\}\\
=&\sum_{i=1}^m\Big\{ |\tau(\p)|^{\frac{\alpha}{2}-1}\langle \onabla_{e_i}d\p(e_i)-d\p(\nabla_{e_i}e_i), \tau(\p)\rangle\Big\}\\
=&|\tau(\p)|^{\frac{\alpha}{2}+1},
 \end{aligned}
 \end{equation}
where the third equality follows from $|\tau(\p)|$ is constant and $\onabla_X\tau(\p)=0, ~~~(X\in \frak{X}(M)).$
 Since $|\tau(\p)|$ is constant and $\int_M|\tau(\p)|^{\alpha}dv_g<\infty$, the function $-\delta\omega$ is also integrable over $M$. 
 From this and $(\ref{Assumption of Gaffney 1})$, we can apply Gaffney's theorem for the $1$-form
 $\omega.$
  Therefore we have
$$0=\int_M(-\delta \omega)dv_g=\int_M|\tau(\p)|^{\frac{\alpha}{2}+1}dv_g,$$
which implies that $\tau(\p)=0.$

\end{proof}

If a biharmonic map $\p:(M,g)\rightarrow (N,h)$ is an isometric immersion, that is, $M$ is a biharmonic submanifold in $N$, we obtain the following result.

\vspace{5pt}

\begin{thm}\label{Th sub mfd}
Let $\p:(M,g)\rightarrow (N,h)$ be a biharmonic isometric immersion from  a complete Riemannian manifold $(M,g)$ into a Riemannian manifold $(N,h)$ with non-positive sectional curvature 
and let $\alpha$ be a real constant satisfying $1<\alpha<\infty$.

If 
$$\int_M|{\bf H}|^{\alpha}dv_g<\infty,$$
then $\phi$ is harmonic.
\end{thm}

\vspace{5pt}

\begin{proof}
By Lemma $\ref{key lem 1}$, we have $\onabla _{X}{\bf H}=0$ for any vector field $X$ on $M$ and  $|{\bf H}|$ is constant.
Since ${\bf H}$ belongs to the normal component of $T_{\p(x)}N$\ $(x\in M)$, $\langle d\p(X),{\bf H}\rangle=0,$ for any vector field $X$ on $M$. 
From these, we obtain 
\begin{align*}
0=&\sum_{i=1}^m\Big\{ e_i\langle d\p(e_i), {\bf H}\rangle 
-\langle d\p(\nabla_{e_i}e_i),{\bf H} \rangle\Big\}\\
=&\sum_{i=1}^m \Big\{ \langle \onabla_{e_i}d\p(e_i), {\bf H}\rangle 
+\langle d\p(e_i), \onabla_{e_i}{\bf H}\rangle 
-\langle d\p(\nabla_{e_i}e_i),{\bf H} \rangle\Big\}\\
=&\sum_{i=1}^m  \langle \onabla_{e_i}d\p(e_i)-d\p(\nabla_{e_i}e_i), {\bf H}\rangle \\
=&m\langle {\bf H}, {\bf H}\rangle.
\end{align*}
Therefore $M$ is minimal, that is, $\p$ is harmonic.
\end{proof}

\qquad\\


\section{Biharmonic submersions into non-positively curved manifolds}\label{sub}

In this section, we apply Theorem $\ref{Th map}$ to submersions.

Wang and Ou showed that a biharmonic Riemannian submersion from a complete Riemannian manifold with constant sectional curvature into a Riemannian surface $(N^2,h)$ is harmonic (cf. \cite{zpwylo1}).
N. Nakauchi, H. Urakawa and S. Gudmundsson applied Theorem $\ref{Th of N-U-G}$ to submersions (cf. \cite{nnhusg1}).

\vspace{5pt}

We first recall harmonic morphisms (cf. \cite{pbjcw1}).

 Let $\p:(M,g)\rightarrow(N,h)$ be a smooth map between Riemannian manifolds, and let $x\in M$. Then $\p$ is called {\em horizontally weakly conformal} at $x$ if either\\
  $(i)$ $d\p_x=0,$ or\\
  $(ii)$ $d\p_x$ maps the horizontal space $\mathcal{H}_x=\{{\rm Ker}(d\p_x)\}^{\perp}$ conformally onto $T_{\p(x)}N,$ such that 
  $$h(d\p_x(X),d\p_x(Y))=\Lambda g(X,Y),~~~~~(X,Y\in \mathcal{H}_x).$$

\vspace{5pt}

 The map $\p$ is called {\em horizontally weakly conformal} on $M$ if it is horizontally weakly conformal at every point of $M$; if further, $\p$ has no critical points, then we call it  a {\em horizontally conformal submersion}.
 Note that if $\p:(M,g)\rightarrow (N,h)$ is a horizontally weakly conformal map and {\rm dim}$M$ $<$ {\rm dim}$N$, then $\p$ is constant.

\vspace{5pt}

If for every harmonic function $f:V\rightarrow \mathbb{R}$ defined on an open subset $V$ of $N$ with $\p^{-1}(V)$ non-empty, the composition $f\circ \p$ is harmonic on $\p^{-1}(V)$, then $\p$ is called a {\em harmonic morphism}. Harmonic morphisms are characterized as follows (cf. \cite{bf1}, \cite{ti1}).

   \vspace{5pt}
   
   \begin{thm}[\cite{bf1}, \cite{ti1}]\label{N-S harmonic morphism}
A smooth map $\p:(M,g)\rightarrow (N,h)$ between Riemannian manifolds is a harmonic morphism if and only if $\p$ is both harmonic and horizontally weakly conformal.
   \end{thm}
   
   \vspace{5pt}
   
Let $\p:(M,g)\rightarrow (N,h)$ be a submersion, then each tangent space $T_xM$ can be decomposed as follows.
\begin{equation}\label{decomposition}
T_xM=\mathcal{V}_x\oplus\mathcal{H}_x,
\end{equation}
where $\mathcal{V}_x={\rm Ker}(d\p_x)$ is the vertical space and $\mathcal{H}_x$ is the horizontal space.
If there exists a positive $C^{\infty}$ function $\lambda$ on $M$ such that, for each $x\in M$,
$$h(d\p_x(X),d\p_x(Y))=\lambda^2(x)g(X,Y),~~~(X,Y\in \mathcal{H}_x),$$
then $\lambda$ is called the {\em dilation}.

When $\p:(M^m,g)\rightarrow (N^n,h)$  $(m>n\geq2)$ is a horizontally conformal submersion,
  the tension field $\tau(\p)$ is given by 
  \begin{equation}\label{tension of sub}
  \tau(\p)=\frac{n-2}{2}\lambda^2d\p\left({\rm grad}_{\mathcal{H}}\left(\frac{1}{\lambda^2}\right)\right)-(m-n)d\p(\hat{\bf H}),
  \end{equation}
  where ${\rm grad}_{\mathcal{H}}\left(\frac{1}{\lambda^2}\right)$ is the $\mathcal{H}$-component of the decomposition according to $(\ref{decomposition})$ of ${\rm grad}\left(\frac{1}{\lambda^2}\right)$, and $\hat{\bf H}$ is the trace of the second fundamental form of each fiber which is given by $\hat{\bf H}=\frac{1}{m-n}\sum_{i=1}^m\mathcal{H}(\nabla_{e_i}e_i)$, where a local orthonormal frame field $\{e_i\}_{i=1}^m$ on $M$ is taken in such a way that $\{e_{ix}| i=1,\cdots, n\}$ belong to $\mathcal{H}_x$ and $\{e_{jx}| j=n+1,\cdots, m\}$ belong to $\mathcal{V}_x$, where $x$ is in a neighborhood in $M$.

 Then following result follows from Theorem $\ref{Th map}$ immediately.
 
  \vspace{5pt}

\begin{prop}
\label{Prop sub}
Let $\p:(M^m,g)\rightarrow (N^n,h)$ $(m>n\geq2)$ be a biharmonic horizontally conformal submersion from  a complete Riemannian manifold $(M,g)$ into a Riemannian manifold $(N,h)$ with non-positive sectional curvature
 and let $\alpha$ be a real constant satisfying $1<\alpha<\infty$.

If
\begin{equation}
\int_M\lambda^2\left|\frac{n-2}{2}\lambda^2{\rm grad}_{\mathcal{H}}\left(\frac{1}{\lambda^2}\right)-(m-n)\hat{\bf H}\right|^{\alpha}_{g}dv_g<\infty,
\end{equation}
and if either $\int_M\lambda^2dv_g<\infty$ or ${\rm Vol}(M,g)=\infty$.
Then, $\p$ is a harmonic morphism.
\end{prop}

\vspace{5pt}

\begin{proof}
By $(\ref{tension of sub})$,
\begin{align*}
\int_M|\tau(\p)|^{\alpha}dv_g
=\int_M\lambda^2\left|\frac{n-2}{2}\lambda^2{\rm grad}_{\mathcal{H}}\left(\frac{1}{\lambda^2}\right)-(m-n)\hat{\bf H}\right|^{\alpha}_{g}dv_g<\infty.
\end{align*}
 Since $\int_M|d\p|^2dv_g=\int_M\lambda^2dv_g$, by using Theorem $\ref{Th map}$, $\p$ is a harmonic map.
 
 Furthermore, since $\p$ is also a horizontally conformal submersion, by Theorem~$\ref{N-S harmonic morphism}$, $\p$ is a harmonic morphism.
\end{proof}

If ${\rm dim}N=2$, Proposition $\ref{Prop sub}$ implies the following corollary.

\vspace{5pt}

\begin{cor}

Let $\p:(M^m,g)\rightarrow (N^n,h)$ $(m>n=2)$ be a biharmonic horizontally conformal submersion from  a complete Riemannian manifold $(M^m,g)$ into a Riemannian manifold $(N^n,h)$ with non-positive sectional curvature 
and let $\alpha$ be a real constant satisfying $1<\alpha<\infty$.

If
\begin{equation}
\int_M\lambda^2\left|\hat{\bf H}\right|^{\alpha}_{g}dv_g<\infty,
\end{equation}
and if either $\int_M\lambda^2dv_g<\infty$ or ${\rm Vol}(M,g)=\infty$.
Then, $\p$ is a harmonic morphism.
\end{cor}






\bibliographystyle{amsbook}

\end{document}